\documentclass[12pt,english]{article}
\usepackage[T1]{fontenc}
\usepackage[latin9]{inputenc}
\usepackage{color}
\usepackage{amsmath}
\usepackage{amssymb}
\usepackage{graphicx}
\usepackage{esint}

\makeatletter

\usepackage{latexsym}
\usepackage{mathrsfs}
\usepackage{graphics}
\usepackage{array}
\usepackage{longtable}
\usepackage{xspace}
\usepackage{amsfonts}
\usepackage{amsbsy}
\usepackage{yfonts}
\usepackage{marvosym}
\usepackage{exscale}
\usepackage{amsthm}
\usepackage{pdflscape}
\usepackage{fullpage}
\usepackage{epsfig}
\usepackage{changebar}

\def\to{{\rightarrow}}

\newcommand{\COMP}{\hbox{C\kern -.58em {\raise .54ex \hbox{$\scriptscriptstyle |$}}
\kern-.55em {\raise .53ex \hbox{$\scriptscriptstyle |$}} }}
\newcommand{\MM}{\hbox{I\kern-.2em\hbox{M}}}
\newcommand{\NN}{\hbox{I\kern-.2em\hbox{N}}}
\newcommand{\RR}{\hbox{I\kern-.2em\hbox{R}}}
\newcommand{\sRR}{{\it \hbox{I\kern-.2em\hbox{R}}}}
\newcommand{\QQ}{\hbox{I\kern-.53em\hbox{Q}}}
\newcommand{\PP}{\hbox{I\kern-.53em\hbox{P}}}
\newcommand{\EE}{\hbox{I\kern-.53em\hbox{E}}}
\newcommand{\ZZ}{{{\rm Z}\kern-.28em{\rm Z}}}
\newcommand{\be}{\begin{equation}}
\newcommand{\ee}{\end{equation}}

\newtheorem{theor}{Theorem}[section]

\newtheorem{prop}{Proposition}[section]
\newtheorem{remar}{Remark}[section]
\newtheorem{defin}{Definition}[section]

\newcommand{\1}{\mbox{1\hspace{-1mm}I}}
\title{Band Control of Mutual Proportional Reinsurance}

\author{John Liu 
\footnote{This work was supported by Hong Kong GRF 5225/07E.}\\
College of Business, City University of Hong Kong, Hong
Kong\\
Michael Taksar
\footnote{This work was supported by the Norwegian Research Council:
Forskerprosjekt ES445026, ``Stochastic Dynamics of Financial Markets.''}\\
Department of Mathematics, University of Missouri, Columbia, Missouri\\
Jiguang Yuan\\
ODDO Options, Hong Kong}

\makeatother

\usepackage{babel}
\begin{document}

\title{Band Control of Mutual Proportional Reinsurance}
\maketitle
\begin{abstract}
In this paper, we investigate the optimization of mutual proportional
reinsurance --- a mutual reserve system that is intended for the collective
reinsurance needs of homogeneous mutual members, such as P\&I Clubs
in marine mutual insurance and reserve banks in the U.S. Federal Reserve.
Compared to general (non-mutual) insurance models, which involve one-sided
impulse control (i.e., either downside or upside impulse) of the underlying
insurance reserve process that is required to be positive, a mutual
insurance differs in allowing two-sided impulse control (i.e., both
downside and upside impulse), coupled with the classical proportional
control of reinsurance. We prove that a special band-type impulse
control $(a,A,B,b)$ with $a=0$ and $a<A<B<b$, coupled with a proportional
reinsurance policy (classical control), is optimal when the objective
is to minimize the total maintenance cost. That is, when the reserve
position reaches a lower boundary of $a=0$, the reserve should immediately
be raised to level $A$; when the reserve reaches an upper boundary
of $b$, it should immediately be reduced to a level $B$.

An interesting finding produced by the study reported in this paper
is that there exists a situation such that if the upside fixed cost
is relatively large in comparison to a finite threshold, then the
optimal band control is reduced to a downside only (i.e., dividend
payment only) control in the form of $(0,0;B,b)$ with $a=A=0$. In
this case, it is optimal for the mutual insurance firm to go bankrupt
as soon as its reserve level reaches zero, rather than to jump restart
by calling for additional contingent funds. This finding partially
explains why many mutual insurance companies, that were once quite
popular in the financial markets, are either disappeared or converted
to non-mutual ones. 
\end{abstract}

\section{Introduction}

Reinsurance has been long investigated as an intrinsic part of commercial
insurance, of which the mainstream modeling framework is profit maximization
with the one-sided impulse control of an underlying reserve process.
There are two types of one-sided impulse control: downside-only impulse
control (such as a dividend payment) with a fixed cost $K^{-}$(e.g.,
Cadenillas \emph{et al.} \cite{Cadenillas2000-}, Hojgaard and Taksar
\cite{hojgaard2004optimal}) and upside-only impulse control (such
as inventory ordering) with a fixed cost $K^{+}$ (e.g., Bensoussan
\emph{et al.} \cite{Bensoussan2005-}, Eisenberg and Schmidli \cite{Eisenberg2010},
Sulem \cite{Sulem1986-125}). In this paper, we examine mutual proportional
reinsurance --- a mutual reserve system that is intended for the collective
reinsurance needs of homogeneous mutual members, such as the P\&I
Clubs in marine mutual insurance (e.g., Yuan \cite{Yuan2008}) and
the reserve banks in the U.S. Federal Reserve (e.g., Dawande \emph{et
al. }\cite{Dawande 2010}). A mutual insurance differs from a general
(non-mutual) insurance in two key dimensions: 1) a mutual system is
not for profit and 2) a mutual reserve involves two-sided impulse
control (i.e., both a dividend refund as a downside impulse to decrease
the reserve with cost $K^{-}$ and a call for funds as an upside impulse
to increase the reserve with cost $K^{+}$). It should be noted that
the reserve process for a general insurance must always be positive
(above zero), and the insurance firm is considered bankrupt as soon
as its reserve falls to zero.

The mutual proportional reinsurance model developed in this paper
is a generalization of the proportional reinsurance models (e.g.,
Cadenillas \emph{et al. }\cite{Cadenillas2000-}, Hojgaard and Taksar
\cite{hojgaard2004optimal}, Eisenberg and Schmidli \cite{Eisenberg2010})
and is modified with the two differing characteristics noted above.
More specifically, the proportional reinsurance rate can be adjusted
in continuous time, and the underlying mutual reserve process is regulated
by a two-sided impulse control in terms of a contingent dividend payment
(i.e., a downside impulse control to decrease the mutual reserve level)
and contingent call for contributions (i.e., an upside impulse control
to increase the mutual reserve level). The corresponding mathematical
problem for mutual proportional reinsurance becomes a two-sided impulse
control system combined with a classical rate control in continuous
time, a problem yet to be posed in insurance research. A problem that
involves a mix of impulse control and classical control is termed
a hybrid control problem in control theory, of which the diculty has
been well noted (e.g., Bensoussan\emph{ }and Menaldi \cite{Bensoussan2000-261},
Branicky and Mitter \cite{Branicky1995-Algorithms}, Abate \emph{et
al.} \cite{Abate2005-Stochastic}).

A pure two-sided impulse control problem (i.e., without a classical
rate control) was investigated by Constantinides \cite{Constantinides1976-1320}
in the form of cash management. Constantinides and Richard \cite{Constantinides1978-620}
showed an optimal two-sided impulse control policy to exist in the
form of a band control, denoted with four parameters as $(a,A;B,b)$
with $a<A\leq B<b$. In other words, when the reserve position reaches
a lower boundary $a$, then the reserve should immediately be raised
to level $A$; when the reserve reaches upper boundary $b$, it should
immediately be reduced to level $B$. For our mutual proportional
reinsurance problem, we specify the corresponding Hamilton-Jacobi-Bellman
(HJB) equation and the associated quasi-variational inequalities (QVI),
from which we analytically solve the optimal value function. We then
prove that a special band-type impulse control $\left(0,A,B,b\right)$
with $a=0$, combined with a proportional reinsurance policy (classical
control), is optimal when the objective is to minimize the total maintenance
cost. An interesting finding reported here is that there exists a
situation such that if the upside fixed cost $K^{+}$ is relatively
large in comparison to a finite threshold $\overline{K^{+}}$, then
the optimal band control is reduced to a downside only (i.e., a dividend
payment only) control in the form of $(0,0;B,b)$ with $a=A=0$. In
this case, it is optimal for the mutual insurance to go bankrupt as
soon as its reserve level falls to zero, rather than to restart by
calling for additional contingent funds. This finding partially explains
why many mutual insurance companies, that were once quite popular
in the financial markets, are either disappeared or converted to non-mutual
ones.

The remainder of the paper is organized as follows. In Section 2,
we formulate the mathematical model and specify the HJB equation and
the QVI of the corresponding stochastic control problem. We solve
the QVI for the optimal value function in Section 3. In Section 3.2,
we characterize and analyze the threshold $\overline{K^{+}}$. In
Section 4, we prove the verification theorem and verify the optimal
control. Finally, we make concluding remarks in Section 5.

\section{The Model}

\subsection{Feasible Control}

The classical Cramer-Lundberg model of an insurance reserve (surplus)
is described via a compound Poisson process: 
\[
D(t)=D(0)+pt-\sum_{i=0}^{N(t)}Y_{i},
\]
 where $D(t)$ is the amount of the surplus available at time $t$,
quantity $p$ represents the premium rate, $N(t)$ is the Poisson
process of incoming claims and $Y_{i}$ is the size of the $i$th
claim. This surplus process can be approximated by a diffusion process
with drift $\mu=p-\lambda E[Y]$ and diffusion coefficient $\sigma=\sqrt{\lambda EY^{2}}$,
where $\lambda$ is the intensity of the Poisson process $N(t)$.
We assume that the insurer always sets $p>\lambda E[Y]$ (i.e. $\mu>0$).
Thus, with no control, the reserve process $X(t)$ is described by
\begin{equation}
X\left(t\right)=X\left(0\right)+\int_{0}^{t}\mu ds+\int_{0}^{t}\sigma dW_{s},\label{dynamics}
\end{equation}
 where $W_{t}$ is a standard Brownian motion.

We start with a probability space $(\Omega,\mathcal{F},P)$, that
is endowed with \textit{information filtration } $\mathcal{F}_{t}$
and a standard Brownian motion $W_{t}$ on $\Omega$ adapted to $\mathcal{F}_{t}$.
Two types of controls are used in this model. The first is related
to the ability to directly control its reserve by raising cash from
or making refunds to members at any particular time. The second is
related to the mutual insurance firm's ability to delegate all or
part of its risk to a reinsurance company, simultaneously reducing
the incoming premium (all or part of which is in this case channeled
to the reinsurance company). In this model, we consider a \textit{proportional
reinsurance} scheme. This type of scheme corresponds to the original
insurer paying $u$ fraction of the original claim. The premium rate
coming to the original insurer is simultaneously reduced by the same
fraction. The reinsurance rate can can be chosen dynamically depending
on the situation.

Mathematically, control $U$ takes a triple form: 
\begin{equation}
U=\left\{ u\left(t\right),\left(\tau_{1},\tau_{2},...\right),\left(\xi_{1},\xi_{2}...\right)\right\} ,\label{contU}
\end{equation}
 where $0\le u(t)\le1$ is a predictable process with respect to ${\mathcal{F}}_{t}$,
the random variables $\tau_{1}<\tau_{2}<...$ constitute an increasing
sequence of stopping times with respect to ${\mathcal{F}}_{t}$, and
$\xi_{i}$ is a sequence of ${\mathcal{F}}_{\tau_{i}}$-measurable
random variables, $i=1,2,...$.

The meaning of these controls is as follows. The quantity $u(t)$
represents the fraction of the claim that the mutual insurance scheme
pays if the claim arrives at time $t$. Suppose that $u(t)=u$ is
chosen at time $t$. Then, in the diffusion approximation (\ref{dynamics}),
drift $\mu$ and diffusion coefficient $\sigma$ are reduced by factor
$u$ (see Cadenillas \emph{et at.} \cite{Cadenillas2006-181}, Hojgaard
and Taksar \cite{hojgaard2004optimal}).

The fact that the process $u(t)$ is adapted to information filtration
means that any decision has to be made on the basis of past rather
than the future information. The stopping times $\tau_{i}$ represent
the times when the $i$th intervention to change the reserve level
is made. If $\xi_{i}>0$, then the decision is to raise cash by calling
the members/clients. If $\xi_{i}<0$, then the decision is to make
a refund. The fact that $\tau_{i}$ is a stopping time and $\xi_{i}$
is ${\mathcal{F}}_{\tau_{i}}$-measurable also indicates that the
decisions concerning when to make a contingent call and how much cash
to raise are made on the basis of only past information. The same
applies to the refund decisions.

Once control $U$ is chosen, the dynamics of the reserve process becomes:
\begin{equation}
X\left(t\right)=X\left(0\right)+\int_{0}^{t}u(s)\mu ds+\int_{0}^{t}u(s)\sigma dW_{s}+\sum_{\tau_{i}\le t}\xi_{i}.\label{dynamics_control}
\end{equation}
 Define the ruin time as 
\begin{equation}
\tau=\inf\{t:X(t)<0\}.\label{tau}
\end{equation}
 Control $U$ is called admissible for initial position $x$ if, for
$X(0)=x$ for any $\tau_{i}\le\tau$, 
\begin{equation}
X(\tau_{i})\ge0;\label{positivity}
\end{equation}
 and if 
\begin{equation}
E\sum_{\tau_{i}\le\tau}e^{-r\tau_{i}}|\xi_{i}|(K^{-}\1_{\xi_{i}<0}+K^{+}\1_{\xi_{i}>0})<\infty.\label{admis}
\end{equation}
 We denote the set of all admissible controls by $\mathcal{U}$.

The meaning of admissibility is as follows. At any time the decision
to make a refund is made, the refund amount cannot exceed the available
reserve. As can be seen in the following, if this condition is not
satisfied, then one can always achieve a cost equal to $-\infty$,
simply by making an infinitely large refund. The second condition
of admissibility is a rather natural technical condition of integrability.

\subsection{Cost Structure and Value Function}

The objective in this model is to minimize the operational cost and
the lost opportunity to invest the money in the market. Cost function
$g$ is defined as 
\begin{equation}
g\left(\xi\right)=K^{+}\1_{\xi>0}+c^{+}\xi^{+}+K^{-}\1_{\xi<0}-c^{-}\xi^{-}.\label{eq:Def-g}
\end{equation}
 Here, $\xi^{+}$ and $\xi^{-}$ denote the positive and negative
components of $\xi$, that is, $\xi^{+}=\max(\xi,0)$ and $\xi^{-}=-\min(\xi,0)$.
The costs associated with refunds are of a different nature. A contingent
call always increases the total cost, whereas a refund decreases it.
However fixed set-up costs $K^{+}>0$ and $K^{-}>0$ are incurred
regardless of of the size of a contingent call or a refund. In addition,
when the call is made and the cash is raised, there is a proportional
cost associated with the amount raised. The constant $c^{+}>1$ represents
the amount of cash that needs to be raised in order for one dollar
to be added to the reserve. If the reserve is used for a refund, then
a part of it may be charged as tax. The constant $0<c^{-}<1$ represents
the amount actually received by the shareholders for each dollar taken
from the reserve.

Given a discount rate $r$, the cost functional associated with the control
$U$ is defined as 
\begin{equation}
C\left(x;U\right)=E_{x}\left\{ \sum_{i=1}^{\infty}g\left(\xi_{i}\right)e^{-r\tau_{i}}\1_{\tau_{i}\le\tau}\right\} .\label{cost-obj3}
\end{equation}
 The objective is to find the value function, 
\begin{equation}
V\left(x\right)=\inf_{U\in\mathcal{U}}C\left(x;U\right),\label{value-funct}
\end{equation}
 and \textit{optimal control} $U^{*}$, such that 
\[
C(x,U^{*})=V(x).
\]

\subsection{Variational Inequalities for the Optimal Value Function\label{sec:QVI}}

For each $0\le u\le1$, define the infinitesimal generator $\mathcal{L}^{u}$.
For any twice continuously differentiable 
function $\phi:\left[0,\infty\right)\mapsto\mathbb{R}$ 
\begin{equation}
\left(\mathcal{L}^{u}\phi\right)\left(x\right)=\frac{1}{2}u^{2}\sigma^{2}\frac{d^{2}\phi\left(x\right)}{dx^{2}}+u\mu\frac{d\phi\left(x\right)}{dx}\label{eq:Lu}.
\end{equation}
 Let $M$ be the \textit{inf-convolution operator}, defined as 
\begin{equation}
\mathcal{M}\phi\left(x\right)=\underset{\xi\neq0}{\inf}\left[g\left(\xi\right)+\phi\left(x+\xi\right)\right].\label{eq:M_oper}
\end{equation}

\begin{defin} The QVI of the control problem are 
\begin{equation}
\mathcal{L}^{u}V-rV\ge 0,\label{eq:QVI 1}
\end{equation}
and
\begin{equation}
\mathcal{M}V\geq V,\label{eq:QVI 2}
\end{equation}
 together with the tightness condition 
\begin{equation}
\left(\mathcal{M}V-V\right)\left(\underset{u\in\left[0,1\right]}{\min}(\mathcal{L}^{u}V-rV)\right)=0.\label{eq:QVI 3}
\end{equation}
 \end{defin}

\section{Solution of the QVI}

\subsection{The HJB Equation in the Continuation Region}

In this model, the application of the control that is related to calls
and refunds results in a jump in the reserve process. This type of
model is considered in the framework of the so-called\textit{impulse control}.
Because we also have a control whose application changes the drift
and the diffusion coefficient of the controlled process, the resulting
mathematical problem becomes a mixed regular-impulse control problem
(e.g., Cadenillas \emph{et al.} \cite{Cadenillas2006-181}). In the
case of a pure impulse control, the optimal policy is of the $(a,A,B,b)$
type, where the four parameters used to construct the optimal control
must be computed as a part of a solution to the problem 
(see Cadenillas and Zapatero \cite{Cadenillas1999-218},
Constantinides and Richard \cite{Constantinides1978-620}, Harrison
and Taylor \cite{Harrison1983-454}, and Paulsen \cite{Paulsen2008}).
Parameters $a$ and $b$ represent the levels at which the intervention
(application of impulses) must be made, whereas $A$ and $B$ stand
for the positions that the controlled process must be in after the
intervention is made. This is a so-called \textit{band-type} policy,
with $(a,A)$ and $(B,b)$ understood as the two bands that determine
the nature of the optimal control. The interval $[a,b]$ is called
the \textit{continuation region}. When the process falls inside the
continuation region, no interventions/impulses are applied. When an
intervention is initiated, the time when the process reaches one of
the boundaries of the continuation region corresponds to one of $\tau_{i}$.

We conjecture that, in our case, the optimal intervention (impulse
control) component of the problem is also of the band type. Moreover,
as the following analysis implicitly shows, we can narrow our search
for the optimal policy to a special band-type control $\left(0,A,B,b\right)$,
where the level $a$ associated with the contingent calls is set to
zero. Therefore, only three of the four band-type policy parameters
remain unknown. After finding these parameters (and determining the
optimal drift/diffusion control in the continuation region), we will
see that the cost function associated with this policy satisfies the
QVI.

The derivation of the value function is similar to \cite{Cadenillas2000-}
and \cite{Cadenillas2006-181} . Suppose that $V\left(x\right)$ satisfies
all of the QVI conditions: \eqref{eq:QVI 1}, \eqref{eq:QVI 2} and
\eqref{eq:QVI 3}. First note that the function $V(x)$ is a decreasing
function of $x$, and thus $V'\le0$. To satisfy \eqref{eq:QVI 3},
for any $x\geq0$, at least one of the two functions on the left side
of the equation should be equal to zero. We conjecture that the value
function has the following structure.

\begin{equation}
V-\mathcal{M}V=0\label{eq:HJB intervention}
\end{equation}
 for $x\in0\cup[b,+\infty)$. Also 
\begin{equation}
\underset{u\in\left[0,1\right]}{\min}(\mathcal{L}^{u}V-rV)=0\label{eq:HJB continuous}
\end{equation}
 for $x\in\left(0,b\right)$.

Assume that $u^{*}\in\left[0,1\right]$ minimizes the function $\mathcal{L}^{u}V-rV$
in foregoing equation. If $V''>0$ then 
\begin{equation}
u^{*}=-\frac{\mu V^{\prime}}{\sigma^{2}V^{\prime\prime}},\label{eq:u_first_order}
\end{equation}
 provided that the right-hand side of (\ref{eq:u_first_order}) belongs
to $(0,1)$. (Note that if $V''(x)=0$, then (\ref{eq:HJB continuous})
cannot be satisfied and we exclude $V''(x)=0$ from consideration.)

Substituting \eqref{eq:u_first_order} into \eqref{eq:HJB continuous},
we get 
\begin{equation}
2r\sigma^{2}V_{1}V_{1}^{\prime\prime}+\mu^{2}\left(V_{1}^{\prime}\right)^{2}=0.\label{eq:HJB 1}
\end{equation}
 The general solution for \eqref{eq:HJB 1} is 
\begin{equation}
V_{1}\left(x\right)=-C_{1}\left(x+C_{2}\right)^{\gamma},\label{eq:General solution for HJB 1}
\end{equation}
 where $C_{1}$ and $C_{2}$ are free constants to be determined later,
and 
\begin{equation}
\gamma=\frac{1}{1+\frac{\mu^{2}}{2r\sigma^{2}}}.\label{gammadef}
\end{equation}
 It is easy to see that $0<\gamma<1$. From \eqref{eq:u_first_order},
we obtain the expression for $u^{*}(x)$ (provided that $C_{1}>0,$
which will be verified later): 
\begin{equation}
u^{*}\left(x\right)=\frac{\mu\left(x+C_{2}\right)}{\sigma^{2}\left(1-\gamma\right)}.\label{eq:u_star(x)}
\end{equation}
 Note that the solution of \eqref{eq:HJB 1} coincides with the solution
of (\ref{eq:HJB continuous}) only in the region where 
\[
0<\frac{\mu}{\sigma^{2}\left(1-\gamma\right)}\left(x+C_{2}\right)<1.
\]
 From this expression, we conjecture that there is a switching point
$x_{0}$ such that $u\left(x\right)=1$ when $x>x_{0}$. As $u^{*}\left(x_{0}\right)=1$,
by virtue of the equation \eqref{eq:u_star(x)}, we obtain the following
expression for $x_{0}$: 
\begin{equation}
x_{0}=\frac{\sigma^{2}\left(1-\gamma\right)}{\mu}-C_{2}.\label{eq:x_0-C_2}
\end{equation}
 For $x>x_{0}$, $u^{*}\left(x\right)=1$; and the corresponding differential
equation becomes 
\begin{equation}
\frac{1}{2}\sigma^{2}V_{2}^{\prime\prime}+\mu V_{2}^{\prime}-rV_{2}=0.\label{eq:HJB_2}
\end{equation}
 The general solution for \eqref{eq:HJB_2} is 
\begin{equation}
V_{2}\left(x\right)=C_{3}e^{\rho_{1}\left(x-x_{0}\right)}+C_{4}e^{-\rho_{2}\left(x-x_{0}\right)},\label{eq:General solution for HJB 2}
\end{equation}
 where 
\begin{align}
\rho_{1}=\frac{\sqrt{\mu^{2}+2r\sigma^{2}}-\mu}{\sigma^{2}}\label{rho1}\\
\rho_{2}=\frac{\sqrt{\mu^{2}+2r\sigma^{2}}+\mu}{\sigma^{2}},\label{rho2}
\end{align}
 with $0<\rho_{1}<\rho_{2}$. Standard arguments show that 
\begin{equation}
V'(x)=-c^{-}\label{Vb}
\end{equation}
for $x\ge b$ (see e.g., Cadenillas \emph{et al.} \cite{Cadenillas2006-181}).
The boundary conditions for the equation are rather tricky. If $0$
and $b$ are the points at which the impulse control (intervention)
is initiated then the boundary conditions at these points become 
\begin{align}
V\left(0\right)=V\left(A\right)+K^{+}+c^{+}A\label{Vl0}\\
V^{\prime}\left(A\right)=-c^{+}\label{VprA}\\
V\left(b\right)=V\left(B\right)+K^{-}-c^{-}\left(b-B\right)\label{Vprb}\\
V^{\prime}\left(B\right)=-c^{-}.\label{VprB}
\end{align}
 However, if bankruptcy is allowed and no intervention is initiated
when the process reaches $0$, then the boundary condition at 0 becomes
straightforward: $V(0)=0$ (see Cadenillas and Zapatero \cite{Cadenillas2000-}
and Cadenillas, \emph{et al.} \cite{Cadenillas2006-181}). In our
case, whether $0$ is the point that corresponds to the intervention
in the form of a contingent call or whether it corresponds to bankruptcy
is not given a priori; rather it is part of the solution to the problem.

We seek the solution by finding  a function $V$ such that 
\begin{align}
V(x)=V_{1}(x),0\le x\le x_{0},\label{V0x0}\\
V(x)=V_{2}(x),x\ge x_{0}\le b,\mbox{ and}\label{Vx0b}\\
V(x)=V(b)-c^{-}(x-b),x\ge b.\label{Vgrb}
\end{align}
 To find the free constants in the expressions for $V_{1}$ an $V_{2}$
and to paste different pieces of the solution together we apply the
\textit{principle of smooth fit} by making the value and the first
derivatives to be continuous at the switching points $x_{0}$ and
$b$, 
\begin{align}
V_{1}(x_{0})=V_{2}\left(x_{0}\right),\label{V1V2}\\
V_{1}^{\prime}\left(x_{0}\right)=V_{2}(x_{0})^{\prime},\label{V1pV2p}\\
V_{2}^{\prime}\left(b\right)=-c^{-},\label{V2pb}
\end{align}
 where $x_{0}$ is defined by \eqref{eq:x_0-C_2}.

(It should be noted that the function $V$, which is constructed from
(\ref{V0x0})-(\ref{Vgrb}) subject to conditions (\ref{Vl0})-(\ref{VprB})
and (\ref{V1V2})-(\ref{V2pb}), corresponds to the case in which
the optimal policy leads to $\tau=\infty$.) We begin by constructing
such a function. The main technique is not to consider the function
$V$ itself, but rather first to construct $V'$.

The form of $V'\left(x\right)$ is shown in Figure \ref{fig:V'(x)}.

\begin{figure}[h]
 \caption{Optimal Policy Parameters}

\label{fig:V'(x)}

\centering{}\includegraphics[scale=0.5]{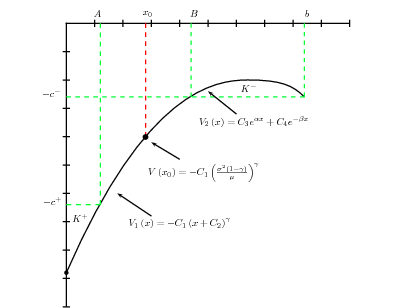} 
\end{figure}

From $u^{*}\left(x_{0}\right)=1$ and \eqref{eq:u_first_order}, we
have $V_{1}^{\prime\prime}\left(x_{0}\right)=-\frac{\mu}{\sigma^{2}}V_{1}^{\prime}\left(x_{0}\right)$.
By the continuity on $V$ and $V'$ at $x_{0}$, and by \eqref{eq:HJB_2},
we have $\frac{\mu}{2r}V_{2}^{\prime}\left(x_{0}\right)=V_{2}\left(x_{0}\right)$.
From this relation and \eqref{eq:General solution for HJB 2}, we
have 
\[
C_{4}=\frac{\rho_{1}\frac{\mu}{2r}-1}{\rho_{2}\frac{\mu}{2r}+1}C_{3}.
\]

Let $\beta=\frac{\rho_{1}\cdot\frac{\mu}{2r}-1}{\rho_{2}\cdot\frac{\mu}{2r}+1}$.
Then, $C_{4}=\beta C_{3}$, and we can write 
\[
V_{2}\left(x;C_{3}\right)=C_{3}e^{\rho_{1}\left(x-x_{0}\right)}+\beta C_{3}e^{-\rho_{2}\left(x-x_{0}\right)}.
\]
We can easily get the inequalities: 
\begin{equation}
-1<\beta<0.\label{beta}
\end{equation}
 From \eqref{eq:x_0-C_2}, we get 
\[
V_{1}\left(x_{0}\right)=-C_{1}\left(x_{0}+C_{2}\right)^{\gamma}=-C_{1}\left(\frac{\sigma^{2}\left(1-\gamma\right)}{\mu}\right)^{\gamma}.
\]
 From $V_{2}\left(x_{0}\right)=C_{3}+\beta C_{3}$, and from the continuity
of $V$ at $x_{0}$, we obtain the expression for $C_{1}$: 
\[
C_{1}=-\left(1+\beta\right)\left(\frac{\sigma^{2}\left(1-\gamma\right)}{\mu}\right)^{-\gamma}C_{3}.
\]

Let $\lambda=-\left(1+\beta\right)\left(\frac{\sigma^{2}\left(1-\gamma\right)}{\mu}\right)^{-\gamma}$.
(Obviously, $\lambda<0$ since $\beta>-1$.) Now, we can write $V_{1}$
in terms of $C_{2}$ and $C_{3}$:
\[
V_{1}\left(x;C_{2},C_{3}\right)=-\lambda C_{3}\left(x+C_{2}\right)^{\gamma}.
\]
 What remains is to determine $C_{2}$ and $C_{3}$. Once these constants
are found, we have $V_{1}\left(x\right)$ and $V_{2}\left(x\right)$,
and thus $V$$\left(x\right)$. Let 
\[
V\left(x,C_{2},C_{3}\right)=\begin{cases}
V_{1}\left(x;C_{2},C_{3}\right)=-\lambda C_{3}\left(x+C_{2}\right)^{\gamma} & \mbox{, for }0<x\le x_{0}\\
V_{2}\left(x;C_{3}\right)=C_{3}e^{\rho_{1}\left(x-x_{0}\right)}+\beta C_{3}e^{-\rho_{2}\left(x-x_{0}\right)} & \mbox{, for }x\ge x_{0}
\end{cases}
\]
 where $x_{0}=\frac{\sigma^{2}\left(1-\gamma\right)}{\mu}-C_{2}$.

Note that if $C_{2}\ge0$ and $C_{3}<0$, then it is easy to show
that for $x>0$, $V_{1}^{\prime}\left(x;C_{2},C_{3}\right)<0$, $V_{1}^{\prime\prime\prime}\left(x;C_{2},C_{3}\right)<0$
and $V_{2}^{\prime}\left(x;C_{2},C_{3}\right)<0$, $V_{2}^{\prime\prime\prime}\left(x;C_{2},C_{3}\right)<0$.
Therefore, $V\left(x,C_{2},C_{3}\right)$ is decreasing on $x>0$
and$V^{\prime}\left(x,C_{2},C_{3}\right)$ is concave on $x>0$. In
the remainder of this section, we find $C_{2}$ and $C_{3}$ and complete
the construction of the function $V$. We do this in an implicit manner
by adopting an auxiliary problem in which no contingent calls are
allowed and by using the optimal value function of that problem to`
construct the function $V$.

Let's consider a slightly different problem in which only those controls
$U$ for which $\xi_{i}$ on the right-hand side of (\ref{contU})
are negative allowed. This problem is similar to that considered
in Cadenillas \emph{et al.} \cite{Cadenillas2006-181}. Let $v(x)$
be the optimal value function for this problem. As was shown in \cite{Cadenillas2006-181},
the function $v$ satisfies the same HJB equation, except for boundary
conditions (\ref{Vl0}) and (\ref{VprA}). These conditions are replaced
by $v(0)=0$.

The same arguments as those above show that we can make the conjecture
that the function $v$ should be sought as a solution to (\ref{eq:aux_HJB1})-\eqref{eq:upper_band_condi}
below. 
\begin{align}
2r\sigma^{2}v\left(x\right)v^{\prime\prime}\left(x\right)+\mu^{2}\left(v^{\prime}\left(x\right)\right)^{2}=0\mbox{ for }0\leq x\leq\tilde{x}_{0}\label{eq:aux_HJB1}\\
\frac{1}{2}\sigma^{2}v^{\prime\prime}\left(x\right)+\mu v^{\prime}\left(x\right)-rv\left(x\right)=0\mbox{ for }\tilde{x}_{0}<x\leq\tilde{b}\label{eq:aux_HJB2}\\
v\left(0\right)=0\label{eq:aux_bound}\\
v^{\prime}\left(\tilde{b}\right)=-c^{-}\label{eq:aux_bound_b}\\
v^{\prime}\left(\tilde{B}\right)=-c^{-}\label{eq:aux_bound_B}\\
v\left(\tilde{b}\right)=v\left(\tilde{B}\right)+K^{-}-c^{-}\left(\tilde{b}-\tilde{B}\right),\label{eq:upper_band_condi}
\end{align}
 where $\tilde{x}_{0}=\frac{\sigma^{2}\left(1-\gamma\right)}{\mu}$.

\subsubsection{A solution to the auxiliary problem}

First note that a general solution to (\ref{eq:aux_HJB1}), (\ref{eq:aux_bound})
is $-cx^{\gamma}$, where $\gamma$ is the same as in (\ref{gammadef})
and $c$ is a free constant, and a general solution to (\ref{eq:aux_HJB2})
is $-a_{1}e^{\rho_{1}(x-\tilde{x}_{0})}-a_{2}e^{\rho_{2}(x-\tilde{x}_{0})}$,
where $\rho_{1}$ and $\rho_{2}$ are the same as in (\ref{rho1}),(\ref{rho2}).

To solve our auxiliary problem we apply the same technique as that
used in Cadenillas \emph{et al.} \cite{Cadenillas2006-181}. We begin
with $H\left(x\right)$, which is defined as follows. 
\begin{equation}
H\left(x\right)=\begin{cases}
-\gamma x^{\gamma-1}, & 0<x\le\tilde{x}_{0}\\
-a_{1}\rho_{1}e^{\rho_{1}\left(x-\tilde{x}_{0}\right)}+a_{2}\rho_{2}e^{-\rho_{2}\left(x-\tilde{x}_{0}\right)}, & x\ge\tilde{x}_{0}.
\end{cases}\label{Hdef}
\end{equation}
 In this expression, constants $a_{1}$ and $a_{2}$ are chosen in
such a way that $-a_{1}\rho_{1}+a_{2}\rho_{2}=-\gamma\tilde{x}_{0}^{\gamma-1}$
and $-a_{1}\rho_{1}^{2}+a_{2}\rho_{2}^{2}=-\gamma(\gamma-1)\tilde{x}_{0}^{\gamma-2}$.
That is, the functions $H$ and $H'$ are continuous at $\tilde{x}_{0}$.
(Note that $H(x)$ is a derivative of $-x^{\gamma}$ on $[0,\tilde{x}_{0}]$
and $H(x)$ is a derivative of $-a_{1}e^{\rho_{1}(x-\tilde{x}_{0})}-a_{2}e^{\rho_{2}(x-\tilde{x}_{0})}$
on $[\tilde{x}_{0},\infty)$.) We next examine the family of functions
$MH(x)$, where $M>0$. We seek $M^{*}$ such that $M^{*}H(x)$ becomes
the derivative of the optimal value function $v$.

\begin{figure}[h]
\begin{centering}
\caption{\label{fig:auxiliary}Solution to the auxiliary problem}

\par\end{centering}

\begin{centering}
\includegraphics[scale=0.7]{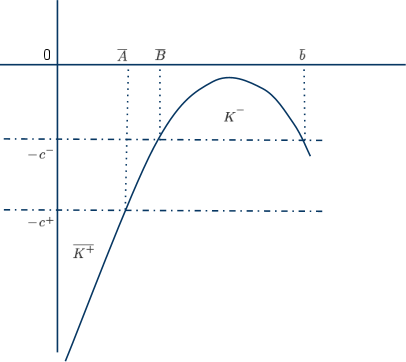} 
\par\end{centering}

\centering{} 
\end{figure}

To this end, we start by finding points $\tilde{b}_{M}$ and $\tilde{B}_{M}$
such that $\tilde{b}_{M}=\max\{x:MH(x)=-c^{-}\}$ and $\tilde{B}_{M}=\min\{x:MH(x)=-c^{-}\}$.
Note that $H$ is a concave function, which is easily checked by differentiation.
Let $\bar{x}=\arg\max H(x)$, the point at which the maximum of $H$
is achieved (it is easy to see that $H'(0)>0$, whereas $\lim_{x\to\infty}H'(x)=-\infty$,
which shows that $\bar{x}$ exists; in view of the fact that $H''(x)<0$,
it is unique). It is obvious by virtue of the concavity of $H$ that,
for any $M\le-c^{-}/H(\bar{x})$, $\tilde{b}_{M}$ and $\tilde{B}_{M}$
exist.

We now consider $I(M)=\int_{\bar{B}^{M}}^{\bar{b}^{M}}\left(MH\left(x\right)+c^{-}\right)\mbox{d}x$.
Informally, $I(M)$ is the area under the graph of $MH(x)$ and above
the horizontal line $y=-c^{-}$ . It is obvious that $I(M)$ is a
continuous function of $M$. For $M=-c^{-}/H(\bar{x})$, we have $\tilde{b}_{M}=\tilde{B}_{M}$;
therefore $I(-c^{-}/H(\bar{x}))=0$. We set $I(M)\to\infty$,
as $M\to0$, because $MH(x)\to0$, whereas $\tilde{b}_{M}\to0$ and
$\tilde{B}_{M}\to\infty$. Therefore, there exists an $M^{*}$ such
that $I(M^{*})=K^{-}$. We also have $\bar{b}=\tilde{b}_{M^{*}}$
and $\bar{B}=\tilde{B}_{M^{*}}$. Let 
\[
H^{*}\left(x\right)=\begin{cases}
M^{*}H(x), & x\le\bar{b},\\
-c^{-}, & x\ge\bar{b}
\end{cases}.
\]
 Then, 
\begin{equation}
v(x)=\int_{0}^{x}H^{*}(y)dy\label{defv1}
\end{equation}
 is the optimal value function of the auxiliary problem (see Figure
\ref{fig:auxiliary}). The proof here is identical to that of a similar
statement in Cadenillas \emph{et al.} \cite{Cadenillas2006-181} and
thus we omit it.

\subsection{The Optimal Value Function for the Original Problem}

We employ the function $H^{*}$ obtained in the previous subsection
to construct the derivative of the optimal value function $V$. The
main idea is to consider $H^{*}(x+S)$ and try to find $S$ such that
$V'(x)=H^{*}(x+S)$. The optimal value function $V$ will then be
sought in the form of $v(x+S)$. To this end, we need the following
proposition. \begin{prop}\label{shiftode} Suppose that $f\left(x\right)$
satisfies (\ref{eq:aux_HJB1}) (satisfies (\ref{eq:aux_HJB2})); then,
for any $S$, the function $f\left(x+S\right)$ satisfies the same
equation on the interval shifted by $S$ to the left. \end{prop}
The proof of this proposition is straightforward.

From(\ref{Hdef}), we can see that $H^{*}$ has a singularity at $0$
with $\lim_{x\downarrow0}H^{*}=-\infty$. The concavity of $H^{*}$
on $[0,\bar{b}]$ implies that $H^{*}$ is increasing on $(0,\bar{x}]$
and decreasing on $[\bar{x},\infty)$ (recall that $H^{*}(x)$ is
constant on $[\bar{b},\infty)$). Therefore, there exists unique 
$0<\bar{A}<\bar{B}$
such that $CH^{*}\left(\bar{A}\right)=-c^{+}$. Define 
\begin{equation}
\overline{K^{+}}=\int_{0}^{\bar{A}}\left(-c^{+}-CH^{*}\left(x\right)\right)\mbox{d}x.
\end{equation}
 Note that $H^{*}$ decreases to $-\infty$ at $0$ at the order of
$x^{\gamma-1}$ (see (\ref{Hdef})); therefore, $H^{*}$ is integrable
at $0$ and, as a result, $\overline{K^{+}}<\infty$.

The qualitative nature of the solution to the original problem depends
on the relationship between $K^{+}$ and $\overline{K^{+}}$; hence,
we divide our analysis into two cases.

\subsubsection{The case of $K^{+}\le\overline{K^{+}}$.}

Consider the following integral 
\[
{\mathcal{J}}(S)=\int_{S}^{\bar{A}}-c^{+}-H^{*}(y)dy
\]
 Geometrically this integral represents the area of a curvilinear
triangle bounded by the lines $y=-c^{+}$, $x=S$ and the graph of
the function $H^{*}(x)$. Obviously, ${\mathcal{J}}(S)$ is a continuous
function of $S$. Because ${\mathcal{J}}(0)=\overline{K^{+}}$ and
$\mathcal{J}(\bar{A})=0$, there exists an $S^{*}$ (see Figure \ref{fig:KpBar_gt_kp})
such that

\begin{equation}
\mathcal{J}(S^{*})=K^{+}.\label{JK+}
\end{equation}

\begin{figure}[h]
 \caption{Determining the parameters for the value function $V(x)$ by utilizing
the solution of the auxiliary problem $H^{*}(x)$ when $K^{+}\le\overline{K^{+}}$.}

\label{fig:KpBar_gt_kp}

\centering{}\includegraphics[scale=0.7]{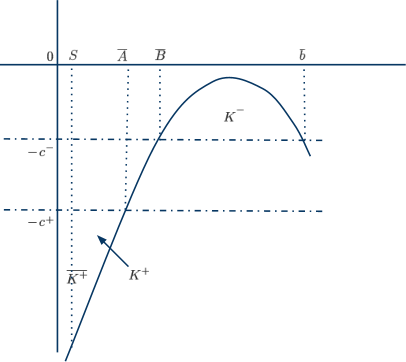} 
\end{figure}

In what follows, we show that $H^{*}(x+S^{*})$ is the derivative
of the solution $V$ to the QVI, inequalities (\ref{eq:QVI 1})-(\ref{eq:QVI 3}).

Let 
\begin{equation}
V(x)=v(x+S^{*}).\label{Vdef1}
\end{equation}
 Also let 
\begin{align}
A=\bar{A}-S^{*},\label{defAf}\\
b=\bar{b}-S^{*},\label{defbf}\\
B=\bar{B}-S^{*},\label{defBf}\\
x_{0}=\tilde{x}_{0}-S^{*},\label{defx0f}\\
x^{*}=\bar{x}-S^{*}.\label{defxst}
\end{align}
 By virtue of Proposition \ref{shiftode}, the function $V$ satisfies
\eqref{eq:HJB 1} on $(0,x_{0}]$, as well as \eqref{eq:HJB_2} on
$[x_{0},b)$. In addition, from (\ref{JK+}), we can see that 
\begin{equation}
V(0)-V(A)\equiv-\int_{S^{*}}^{\bar{A}}H^{*}(y)dy=c^{+}A+K^{+}.\label{Vvarequal}
\end{equation}
 From the construction of the function $v$, we can also see that
\begin{equation}
V(b)-V(B)\equiv v(\bar{b})-v(\bar{B})=-c^{-}(b-B)+K^{-}.\label{VvarequalB}
\end{equation}
 \begin{theor}\label{hjbproof} The Function $V$ given by (\ref{VvarequalB})
is a solution to the QVI (\ref{eq:QVI 1})-(\ref{eq:QVI 3}). \end{theor}
The proof of this theorem is divided into several propositions. \begin{prop}\label{qvicont0b}
The function $V$ satisfies (\ref{eq:HJB continuous}) on $[0,b)$.
\end{prop} \begin{proof} 1$^{\circ}$. From the construction of
the function $v$, we have $-\frac{\mu v^{\prime}}{\sigma^{2}v^{\prime\prime}}\le1$
on $(0,\tilde{x}_{0}]$. Consequently, 
\begin{equation}
-\frac{\mu V^{\prime}}{\sigma^{2}V^{\prime\prime}}<1\label{Vless1}
\end{equation}
 on $(0,x_{0}]$. As $V$ satisfies (\ref{eq:HJB 1}), it also satisfies
(\ref{eq:HJB continuous}) because these two equations are equivalent
whenever (\ref{eq:HJB continuous}) holds.

2$^{\circ}$. To prove that (\ref{eq:HJB continuous}) holds for $x\ge x_{0}$,
it is sufficient to show that for each $0\le u\le1$, 
\begin{equation}
\mathcal{L}^{u}V-rV\equiv\frac{1}{2}\sigma^{2}u^{2}V_{2}^{\prime\prime}+\mu uV_{2}^{\prime}-rV_{2}=0\ge0\label{HJBinequal}
\end{equation}
 because, in view of Proposition \ref{shiftode}, the function $V$
satisfies (\ref{eq:HJB_2}) (that is, $\mathcal{L}^{1}V-rV=0$). By
subtracting (\ref{HJBinequal}) from the left hand side of (\ref{eq:HJB_2}),
we can see that (\ref{HJBinequal}) is equivalent to 
\begin{equation}
\frac{1}{2}\sigma^{2}(1-u^{2})V_{2}^{\prime\prime}(x)+\mu(1-u)V_{2}^{\prime}(x)\le0\label{equivHJBineq}
\end{equation}
 for $x\in[x_{0},b)$. In view of the continuity of the first and
second derivatives, and in view of the fact that (\ref{eq:HJB continuous})
holds on $(0,x_{0}]$, we know that (\ref{equivHJBineq}) is true
for $x=x_{0}$. As both $V'<0$ and $V''<0$ on the left-hand side
of (\ref{equivHJBineq}) are decreasing, therefore is nonpositive for all
$x\ge x_{0}$. \end{proof}

Note that $\mathcal{M}\phi(x)=\min(\mathcal{M}_{1}\phi(x),\mathcal{M}_{2}\phi(x))$,
where 
\[
\mathcal{M}_{1}\phi(x)=\inf_{\xi>0}[K^{+}+c^{+}\xi+\phi(x+\xi)];\,\,\,\mathcal{M}_{2}\phi(x)=\inf_{\xi>0}[K^{-}+c^{-}\xi+\phi(x-\xi)]
\]

\begin{prop}\label{qvim1m2} For each $x\ge0$, we have 
\begin{equation}
\mathcal{M}V(x)\ge V(x).\label{MVgeneral}
\end{equation}
 If $x=0$, or if $x\ge b$, then 
\begin{equation}
\mathcal{M}V(x)=V(x).\label{MVequal}
\end{equation}
 \end{prop} \begin{proof} 1$^{\circ}$. We first prove, that $\mathcal{M}_{1}V(x)\ge V(x)$.
Suppose that $x<A$. The functional $V(x+\xi)+c^{+}\xi$ is continuously
differentiable. By construction, $V'(x)=H^{*}(x+S^{*})$ is increasing
on $(0,x^{*})$ and decreasing on $[x^{*},\infty)$ with $V'(x)\equiv-c^{-}$
for $x\ge b$. Therefore, the point $A$ is the only point $y$ such
that $V'(y)\equiv H^{*}(y_{S}^{*})=-c^{+}$. Because $(H^{*})'(A)<0$,
we can see that $A=\arg\min_{x}V(x+\xi)+c^{+}\xi$. Therefore, for
$x\le A$, 
\[
\mathcal{M}_{1}V(x)=V(A)+c^{+}(A-x)+K^{+}.
\]
 Also, $\mathcal{M}_{1}V(x)-V(x)=V(A)-V(x)+c^{+}(A-x)+K^{+}\ge0$
iff 
\[
-c^{+}(A-x)-V(A)-V(x)\equiv\int_{x}^{A}[-c^{+}-V(y)]dy\le K^{+}.
\]
 However, 
\[
\int_{x}^{A}[-c^{+}-V(y)]dy\le\int_{0}^{A}[-c^{+}-V(y)]dy\equiv\int_{S^{*}}^{\bar{A}}[-c^{+}-H^{*}(z)]dz=K^{+}.
\]
 This proves that $M_{1}V(x)\ge V(x)$ for all $x\le A$ and also
shows that 
\begin{equation}
\mathcal{M}_{1}V(0)=V(0).\label{MV0}
\end{equation}
 Because $V'(y)>-c^{+}$ for $y>A$, we know that for any $x>A$ the
function $V(x+\xi)+c^{+}\xi$ is an increasing function of $\xi>0$.
Therefore, the minimum in the expression for $\mathcal{M}_{1}$ is
attained for $\xi=0$; as a result, $\mathcal{M}_{1}(x)=V(x)+K^{+}>V(x)$.

2$^{\circ}$. Consider $V(x-\xi)-c^{-}\xi$ for $b>x\ge B$. Because
$V'(y)>-c^{-}$ for $y>B$ and $V'(y)<c^{-}$ for $y<B$, we can see
that $V(x-\xi)-c^{-}\xi$ has a unique minimum at $x-\xi=B$. Therefore,
$\mathcal{M}_{2}V(x)=V(B)-c^{-}(x-B)+K^{-}$. Thus, $\mathcal{M}_{2}V(x)-V(x)\ge0$
iff 
\[
V(x)-V(B)+c^{-}(x-B)\equiv\int_{B}^{x}V'(y)+c^{-}dy\le K^{-}.
\]
 The foregoing inequality always holds true because, by construction,
\[
K^{-}=\int_{\bar{B}}^{\bar{b}}H^{*}(y)+c^{-}dy\equiv\int_{B}^{b}V'(y)+c^{-}dy.
\]
 For $x\ge b$, we note, that $V(x)-V(b)=-c^{-}(x-b)$, and hence 
\begin{equation}
\mathcal{M}_{2}V(x)-V(x)=\mathcal{M}_{2}V(b)-V(b)=V(B)-V(b)-c^{-}(b-B)+K^{-}=0.\label{V2b}
\end{equation}
 \end{proof} To complete the proof of Theorem \ref{hjbproof}, we
need only show the following. \begin{prop}\label{qvicontgrb} If
$x>b$, then (\ref{eq:QVI 1}) holds. \end{prop} \begin{proof} It
is sufficient to show that for any $0\le u\le1$ and any $x>b$, 
\begin{equation}
f(x)\equiv\frac{1}{2}u^{2}\sigma^{2}\frac{d^{2}V(x)\left(x\right)}{dx^{2}}+u\mu\frac{dV(x)\left(x\right)}{dx}-rV(x)\ge0.\label{qvigrb}
\end{equation}
 From Proposition \ref{qvicont0b} we know that (\ref{qvigrb}) is
true for $x=b-$. As $V''(b-)<0$, we obtain $\left(\mathcal{L}^{u}V\right)\left(b+\right)-rV(b+)=u\mu c^{-}-rV(b)\ge0$.
For any $x>b$, we have $\left(\mathcal{L}^{u}V\right)\left(x\right)-rV(x)=u\mu c^{-}-rV(x)>c^{-}-V(b)$
because $V$ is a decreasing function. \end{proof} This completes
the proof of Theorem \ref{hjbproof}.

\subsubsection{The case of $K^{+}>\overline{K^{+}}$}

When $K^{+}>\overline{K^{+}}$ we cannot find any $S^{*}$ such that
(\ref{JK+}) is satisfied. In this case, we set $V(x)=v(x)$ (that
is, we have $S^{*}=0$, which corresponds to $A=0$). \begin{theor}\label{verifineqtheor}
If $K^{+}>\overline{K^{+}}$, then $V(x)=v(x)$ is a solution to the
QVI (\ref{eq:QVI 1})-(\ref{eq:QVI 3}). \end{theor} To prove this
theorem, it is sufficient to prove Propositions \ref{qvicont0b}-\ref{qvicontgrb}.
The proofs of Propositions \ref{qvicont0b} and \ref{qvicontgrb}
are identical to the case of $K^{+}\leq\overline{K^{+}}$, whereas
that of Proposition \ref{qvim1m2} requires a slight modification.
\begin{prop}\label{qvim1m2m} For each $x\ge0$, 
\begin{equation}
\mathcal{M}V(x)\ge V(x).\label{MVgeneral1}
\end{equation}
 If $x\ge b$, then 
\begin{equation}
\mathcal{M}V(x)=V(x).\label{MVequal1}
\end{equation}
 \end{prop} \begin{proof} The proof that $\mathcal{M}_{2}V(x)\ge V(x)$
for all $x$ and that $\mathcal{M}_{2}V(x)=V(x)$ for $x\ge b$ is
the same as that in Proposition \ref{qvim1m2}.

If $x\le A$, then $\mathcal{M}_{1}V(x)\ge V(x)$ is equivalent to
\[
-c^{+}(A-x)-V(A)-V(x)\equiv\int_{x}^{A}[-c^{+}-V(y)]dy\le K^{+}.
\]
 The foregoing inequality is always true because 
\[
\int_{x}^{A}[-c^{+}-V(y)]dy\le\int_{0}^{A}[-c^{+}-V(y)]dy=\overline{K^{+}}<K^{+}
\]
 by assumption.\\
 If $x>A$, then $\mathcal{M}_{1}V(x)\ge V(x)$ still holds, due to
the same argument as that in Proposition \ref{qvim1m2}. \end{proof}
\begin{remar} In the case of $K^{+}>\overline{K^{+}}$ we have $\mathcal{M}V(x)=V(x)$
only for $x\ge b$, whereas $\mathcal{M}V(0)>V(0)$ in contrast to
the case when $K^{+}\le\overline{K^{+}}$. Also, when $K^{+}>\overline{K^{+}}$,
we have $V(0)=0$, whereas $V(0)<0$ if $K^{+}\le\overline{K^{+}}$.
Equivalently, in the case that the fixed cost to call for additional
funds is relatively large (i.e., $K^{+}>\overline{K^{+}})$ , the
optimal band control is reduced to $(0,0;B,b)$ with $a=A=0$. That
is, as soon as the reserve reaches zero, it becomes optimal for the
mutual insurance firm to go bankrupt, rather than to be restarted
by calling for additional funds. \end{remar}

\section{Verification Theorem and the Optimal Control}

\begin{theor}\label{the:verification} If $V$ is a solution to QVI
(\ref{eq:QVI 1})-(\ref{eq:QVI 3}), then for any control $U$, 
\begin{equation}
V(x)\le C(x,U).\label{ineqtheorem}
\end{equation}
 \end{theor} \begin{proof} We prove this inequality when $K^{+}\le\overline{K^{+}}$.
In this case, $V'(0)=v'(S^{*})>-\infty$. Let $U$ be any admissible
control defined by (\ref{contU}) and process $X(t)$ be the corresponding
surplus process (\ref{dynamics_control}), with $X(0)=x$. Let $\tau$
be its ruin time given by (\ref{tau}). Let $\tau(t)=\tau\wedge t$
and $\tau_{i}(t)=\tau_{i}\wedge\tau\wedge t$. Then,

\begin{eqnarray}
e^{-r\tau(t)}V(X(\tau(t)))-V(x) & = & \sum_{i=1}^{\infty}\left[e^{-r\tau_{i}(t)}V\left(X\left(\tau_{i}(t)-\right)\right)-e^{-r\tau_{i-1}}(t)V\left(X\left(\tau_{i-1}(t)\right)\right)\right]\nonumber \\
 &  & +\sum_{i=1}^{\infty}e^{-r\tau_{i}(t)}\left[V\left(X\left(\left(\tau_{i}(t)\right)\right)\right)-V\left(X\left(\tau_{i}(t)-\right)\right)\right].\label{exprepres}
\end{eqnarray}

By convention, $\tau_{0}=0$. In view of (\ref{eq:QVI 2}), we have
$|V(x)-V(y)|(c^{+}|y-x|+K^{+})\1_{y<x}+(K^{-}-c^{-}|y-x|)\1_{y>x}$.
Therefore, $|V(X(\tau_{i}(t)))-V(X(\tau_{i}(t)-))|(c^{+}\xi_{i}^{+}+K^{+})\1_{\xi_{i}>0}+(K^{-}-c^{-}\xi_{i}^{-})\1_{\xi<0}$
and

\begin{eqnarray}
\left|\sum_{i=1}^{\infty}e^{-r\tau_{i}(t)}\left[V\left(X\left(\tau_{i}(t)\right)\right)-V\left(X\left(\tau_{i}(t)-\right)\right)\right]\right| & \le & \sum_{i=1}^{\infty}e^{-r\tau_{i}(t)}\left|V\left(X\left(\tau_{i}(t)\right)\right)-V\left(X\left(\tau_{i}(t)-\right)\right)\right|\label{diffVineq}\\
 & \le & \sum_{i=1}^{\infty}e^{-r\tau_{i}(t)}\left[(c^{+}\xi_{i}^{+}+K^{+})\1_{\xi_{i}>0}+(K^{-}-c^{-}\xi_{i}^{-})\1_{\xi_{i}<0}\right]\nonumber \\
 & \le & \sum_{\tau_{i}\le\tau}e^{-r\tau_{i}}\left[(c^{+}\xi_{i}^{+}+K^{+})\1_{\xi_{i}>0}+(K^{-}-c^{-}\xi_{i}^{-})\1_{\xi_{i}<0}\right]\nonumber 
\end{eqnarray}

In view of (\ref{admis}), this implies that for any $t$ the second
sum in (\ref{exprepres}) is bounded by the same integrable random
variable independent of $t$.

On $[\tau_{i-1},\tau_{i}]$, process $X(t)$ is continuous, and we
can apply Ito's formula to get

\begin{eqnarray}
e^{-r\tau_{i}(t)}V\left(X\left(\tau_{i}(t)-\right)\right)-e^{-r\tau_{i-1}}(t)V\left(X\left(\tau_{i-1}(t)\right)\right) & = & \int_{\tau_{i-1}(t)}^{\tau_{i}(t)}e^{-rs}\left[\mathcal{L}^{u(s)}V\left(X(s)\right)-rV\left(X(s)\right)\right]ds\nonumber \\
 &  & +\int_{\tau_{i-1}(t)}^{\tau_{i}(t)}e^{-rs}V'\left(X(s)\right)u(s)\sigma dW(s).\label{ineqVcontseg}
\end{eqnarray}

From this equation, using (\ref{admis}) and standard but rather tedious
arguments, we can deduce that the first sum in (\ref{diffVineq})
for all $t$ is also bounded by the same integrable random variable.
Similar arguments show that

\begin{equation}
\lim_{t\to\infty}E[e^{-r\tau(t)}V(X(\tau(t)))\1_{\tau=\infty}]=0\label{limiinfty}
\end{equation}
 and 
\begin{equation}
\lim_{t\to\infty}E[e^{-r\tau(t)}V(X(\tau(t)))]=E[e^{-r\tau}V(X(\tau))\1_{\tau<\infty}].\label{limiinfty1}
\end{equation}
 (see Cadenillas \emph{et al.} \cite{Cadenillas2006-181}). Note that
the second integral on the right-hand side of (\ref{ineqVcontseg})
is a martingale whose expectation vanishes. However, in view of (\ref{eq:QVI 1}),
the integrand in the first integral of (\ref{ineqVcontseg}) is nonnegative.
Therefore, 
\begin{equation}
E[e^{-r\tau_{i}(t)}V(X(\tau_{i}(t)-))-e^{-r\tau_{i-1}}(t)V(X(\tau_{i-1}(t)))]>0,\label{ineqVcontsegpos}
\end{equation}
 and, taking into account the dominated convergence theorem, we can
see that the expectation of the first sum on the right-hand side of
(\ref{exprepres}) is nonnegative.

From (\ref{eq:QVI 2}), we can see that $V(X(\tau_{i}))-V(X(\tau_{i}-))\equiv V(X(\tau_{i}-)+\xi_{i})-V(X(\tau_{i}-))\ge-[(K^{+}+c^{+}\xi_{i})\1_{\xi_{i}>0}+(K^{-}-c^{-}\xi_{i})\1_{\xi_{i}<0}]$.
Substituting this inequality into (\ref{exprepres}) and taking expectations
of both sides, we obtain 
\begin{equation}
E[e^{-r\tau(t)}V(X(\tau(t)))]-V(x)\ge-E\sum_{\tau_{i}\le\tau,\tau_{i}\le t}e^{-r\tau_{i}}[(K^{+}+c^{+}\xi_{i}^{+})\1_{\xi_{i}>0}+(K^{-}-c^{-}\xi_{i}^{-})\1_{\xi_{i}<0}]\label{ineqprelim}
\end{equation}
 Letting $t\to\infty$, and employing (\ref{limiinfty1}) 
and the monotone convergence theorem on, we get 
\begin{equation}
E[e^{-r\tau}V(X(\tau))1_{\tau<\infty}]-V(x)\ge-E\sum_{\tau_{i}\le\tau}e^{-r\tau_{i}}[(K^{+}+c^{+}\xi_{i}^{+})\1_{\xi_{i}>0}+(K^{-}-c^{-}\xi_{i}^{-})\1_{\xi_{i}<0}]=-C(x,U)\label{ineqverfinal}
\end{equation}
 As $V(X(\tau))\le0$, on $\{\tau<\infty\}$,
the  inequality (\ref{ineqverfinal}) implies \eqref{ineqtheorem}
\end{proof} \begin{remar} For the expectation of the stochastic
integral on the right hand side of (\ref{ineqVcontseg}) to vanish,
it is sufficient for its integrand to be bounded. In particular, it
is sufficient for $V'$ to be bounded. This is the case when $K^{+}<\overline{K^{+}}$.
When $K^{+}\ge \overline{K^{+}}$, the function $V'$ has a singularity
at 0. We can, however, apply the same technique, first replacing $\tau$
by $\tau^{\epsilon}=\inf\{t:X(t)<\epsilon\}$ and then passing to
a limit as $\epsilon\to0$. This will yield inequality (\ref{ineqprelim}),
which is all we need for the proof of Theorem \ref{the:verification}
\end{remar}

Let $\mathcal{H}^{*}(x)=\arg\min_{u}\{\mathcal{L}^{u}V(x)-rV(x)\}$,
that is 
\begin{equation}
\mathcal{L}^{{\mathcal{H}}^{*}(x)}V(x)-rV(x)=0\label{equalL}
\end{equation}
 From (\ref{defv1}),(\ref{Hdef}), (\ref{Vdef1}), and (\ref{defx0f}),
we can see that 
\[
\mathcal{H}^{*}(x)=\min\{x/(x_{0}+S^{*}),1\}.
\]
 Consider the process defined as 
\begin{equation}
X^{*}(t)=X(\tau_{\tau_{i}})+\int_{\tau_{i}}^{t}\mu\mathcal{H}^{*}(X^{*}(s))ds+\int_{\tau_{i}}^{t}\sigma\mathcal{H}^{*}(X^{*}(s))dW_{s}\label{defoptproc}
\end{equation}
 $i=0,1,...$, with $\tau_{0}\equiv0$, $X(\tau_{0})\equiv X(0)=x$,
and $\tau_{i}$ and $X(\tau_{i})$ are defined below sequentially.

If $\overline{K^{+}}<K^{+}$(that is, no $S^{*}$ exists, such that
(\ref{JK+}) holds), then 
\begin{align}
\tau_{i}=\inf\{t:X^{*}(t)\ge b\}\label{tauforv}\\
X^{*}(\tau_{i})=B,\,\,\,\xi_{i}=b-B.\label{xtauforv}
\end{align}
 If $\overline{K^{+}}\ge K^{*}$ (that is, there exists an $S^{*}$,
such that (\ref{JK+}) holds), then 
\begin{align}
\tau_{i}=\inf\{t:X^{*}(t)\not\in[0,b]\}\label{tauforV}\\
X^{*}(\tau_{i})=\begin{cases}
B,\,\,\,\xi_{i}=b-B & X^{*}(\tau_{i}-)=b,\\
A,\,\,\,\xi_{i}=A & X^{*}(\tau_{i}-)=0.
\end{cases}\label{xtauforV}
\end{align}
 \begin{remar} Informally, if $\overline{K^{+}}<K^{+}$, then process
$X^{*}(s)$ is a continuous diffusion process with a drift and diffusion
coefficient of, $\mu\mathcal{H}^{*}(X^{*}(s))$ and $\sigma\mathcal{H}^{*}(X^{*}(s))$,
respectively, until the times of intervention. The times of intervention
in this case are the times at which this diffusion process hits the level
$b$, which are associated with the refunds of the constant
amount of $b-B$. The time when $0$ is hit is the ruin time.

When $\overline{K^{+}}\ge K^{+}$, the process $X^{*}(s)$ is a continuous
diffusion process with the same drift and diffusion coefficients as
above, between the times of intervention. The intervention times
are the times at which this process reaches either $0$ or $b$. At
point $0$, the control is set to displace the process to point $A$,
which corresponds to raising cash (making a call to shareholders)
in the amount of $A$. Reaching the level $b$ results in the displacement
of the process to the point $B$ which, corresponds to making a refund
in the amount of $b-B$. \end{remar} \begin{theor} \textbf{(The
verification theorem)} Let $U^{*}$ be the control described by (\ref{equalL})
and (\ref{tauforv})-(\ref{xtauforV}). Then, 
\[
C(x,U^{*})=V(x)=\inf_{U\in{\mathcal{U}}}C(x,U)
\]
 \end{theor} \begin{proof} In view of (\ref{ineqtheorem}), it is
sufficient to show that 
\begin{equation}
C(x,U^{*})=V(x).\label{verfinal}
\end{equation}
 Equality (\ref{equalL}) shows that 
\begin{equation}
\mathcal{L}^{{\mathcal{H}}^{*}(X^{*}(s))}V(X^{*}(s))-rV(X^{*}(s))=0.\label{eq0inverif}
\end{equation}
 From Propositions \ref{qvim1m2} and \ref{qvim1m2m}, we know that
$V(b)-V(B)=K^{-}-c^{-}(b-B)$ and if $\overline{K^{+}}\ge K^{+}$
then $V(0)-V(A)=K^{+}+c^{+}A$.
Thus, we can repeat the arguments in the proof of Theorem \ref{verifineqtheor}
and see that, for $U=U^{*}$, all of the inequalities are tight. As
a result, we obtain 
\begin{equation}
E[e^{-r\tau}V(X(\tau))\1_{\tau<\infty}]-V(x)=-E\sum_{\tau_{i}\le\tau}e^{-r\tau_{i}}[(K^{+}+c^{+}\xi^{+})\1_{\xi_{i}>0}+(K^{-}-c^{-}\xi^{-})\1_{\xi_{i}<0}]=C(x,U^{*}).\label{ineqverfinalver}
\end{equation}
 Because we know that, when $\overline{K^{+}}\ge K^{+}$, we have
$\1_{\tau<\infty}=0$ and, when $\overline{K^{+}}<K^{+}$, the function
$V$ satisfies $V(0)\equiv v(0)=0$, the first term on the left-hand
side of (\ref{ineqverfinalver}) vanishes, and we get \eqref{verfinal}.
\end{proof}

\section{Conclusions}

The optimal policy in this model has several interesting nontrivial
features. The fact that calls should be made only when there is no
possibility of waiting any longer (that is when the reserve reaches
zero) is supported by intuition. However, the qualitative structure 
of the optimal
policy and its dependence on the model parameters are not as obvious.

It turns out that it is always optimal to pay dividends, no matter
what the costs associated with such payments are. However, raising
cash may not be optimal when the initial set-up cost is too high.
Quantity $\overline{K^{+}}$, which determines the threshold for set-up
cost $K^{+}$, such that if the cost is higher than this threshold
then it is optimal to allow ruin, is in itself determined via an auxiliary
problem with a one-sided impulse control. Although there is no closed-form
expression for the quantity $\overline{K^{+}}$, it can be determined
in an algorithmic manner prior to solving the optimal control problem
for the mutual insurance company.

There is one rather curious feature of the optimal solution when $K^{+}=\overline{K^{+}}$.
As our analysis shows, in this case, $S^{*}=0$ and $V=v$, the
same as is the case when $K^{+}>\overline{K^{+}}$. However, from the
construction of the optimal policy, we can see that the two band-type
policy is optimal in this case as well. In this borderline case, we
thus have two optimal policies, one for which $\tau=\infty$ with
the lower band equal to $(0,A)$ and one for which reaching $0$ corresponds
to ruin and for which $\tau<\infty$. This is a rather unique feature
of this particular problem that has not been observed previously.

A natural question arises: what if ruin is explicitly disallowed,
and we must find an optimal policy from among those for which $\tau=\infty$.
As can be seen from our analysis, we find a solution to this problem
for the case of $\overline{K^{+}}\ge K^{+}$. However, when this inequality
does not hold, then of the stochastic control technique and the HJB
equation used in this paper do not work. Another approach should be
developed, as can be seen indirectly in the work of Eisenberg \cite{Eisenberg2011Optimal}
and Eisenberg and Schmidli \cite{Eisenberg2010}, where a similar
(although not identical) problem is considered for the case of a surplus
process modeled via the classical Cramer-Lundberg model. This constitutes
an interesting and challenging problem for future study, the nature
of whose solution is not obvious at this time.



\begin{thebibliography}{References}
\bibitem{Abate2005-Stochastic}Abate, A., A. D. Ames, S. Sastry. 2005.
Stochastic approximations for hybrid systems, \emph{Proceedings of
the 24th American Control Conference}, Portland, OR, 1557-1562.

\bibitem{Bensoussan2005-}Bensoussan, A., R.H. Liu, S.P. Sethi. 2006.
Optimality of an (s, S) policy with compound Poisson and diffusion
demands: a quasi-variational inequalities approach, \emph{SIAM Journal
on Control and Optimization.} 44(5), 1650-1676.

\textcolor{red}{\bibitem{Bensoussan2000-261}}Bensoussan, A., J. L.
Menaldi. 2000. Stochastic hybrid control, \emph{Journal of Mathematical
Analysis and Applications}, 249, 261-288.

\bibitem{Branicky1998-31}Branicky, M.S., V.S. Borkar, S.K. Mitter.
1998. A unified framework for hybrid control: model and optimal control
theory, \emph{IEEE Transactions on Automatic Control,} 43(1), 31-45.

\bibitem{Branicky1995-Algorithms}Branicky, M.S., S.K. Mitter. 1995.
Algorithms for optimal hybrid control, \emph{Proceedings of the 34th
IEEE Conference on Decision and Control.}

\bibitem{Cadenillas1999-218}Cadenillas, A., F. Zapatero. 1999. Optimal
central bank intervention in the foreign exchange market, \emph{Journal
of Economic Theory,} 87(1), 218-242.

\bibitem{Cadenillas2006-181}Cadenillas, A., T. Choulli, M. Taksar,
L. Zhang. 2006. Classical and impulse stochastic control for the optimization
of the dividend and risk policies of an insurance firm, \emph{Mathematical
Finance.} 16(1), 181-202.

\bibitem{Cadenillas2000-}Cadenillas, A., F. Zapatero. 2000. Classical
and impulse stochastic control of the exchange rate using interest
rate and reserves, \emph{Mathematical Finance}, 10, 141-156.

\bibitem{Constantinides1976-1320}Constantinides, G. M. 1976. Stochastic
cash management with fixed and proportional transaction costs, \emph{Management
Science}, 22, 1320-1331.

\bibitem{Constantinides1978-620}Constantinides, G.M., S.F. Richard.
1978. Existence of optimal simple policies for discounted-cost inventory
and cash management in continuous time, \emph{Operations Research,}
26(4), 620-636.

\bibitem{Dawande 2010}Dawande, M., M. Mehrotra, V. Mookerjee, C.
Srikandarajah. 2010. An analysis of coordination mechanism for the
U.S. cash supply chain. \emph{Management Science}, 56(3), 553 - 570.

\bibitem{Eisenberg2011Optimal}Eisenberg, J. 2011. Optimal control
of capital injections by reinsurance and investments, \emph{working
paper}.

\bibitem{Eisenberg2010}Eisenberg, J., H. Schmidli. 2010. Minimizing
expected discounted capital injections by reinsurance in a classical
risk model, \emph{Scandinavian Actuarial Journal}, forthcoming (first
published on March 19, 2010).

\bibitem{Harrison1983-454}Harrison, J.M., T.M. Sellke, A.J. Taylor.
1983. Impulse control of Brownian motion, \emph{Mathematics of Operations
Research,} 8(3), 454-466.

\bibitem{hojgaard2004optimal}Hojgaard, B., M. Taksar. 2004. Optimal
dynamic portfolio selection for a corporation with controllable risk
and dividend distribution policy, \emph{Quantitative Finance}, 4(3),
256-265.

\bibitem{lokka2008optimal}L\{\textbackslash{}o\}kka, A,. M. Zervos.
2008. Optimal dividend and issuance of equity policies in the presence
of proportional costs, \emph{Insurance: Mathematics and Economics},
42(3), 954-961.

\bibitem{Paulsen2008}Paulsen, J. 2008. Optimal dividend payments
and reinvestments of diffusion processes with both fixed and proportional
costs, \emph{SIAM Journal on Control and Optimization,} 47(5), 2201-2226.

\bibitem{schmidli2001optimal}Schmidli, H. 2001. Optimal proportional
reinsurance policies in a dynamic setting, \emph{Scandinavian Actuarial
Journal}, Taylor \& Francis, 1, 55-68.

\bibitem{Sulem1986-125} Sulem, A. 1986. A solvable one-dimensional
model of a diffusion inventory system, \emph{Mathematics of Operations
Research}, 11, 125-133.

\bibitem{Yuan2008}Yuan, J. 2008. Computational optimization of mutual
insurance systems: a quasi-variational inequality approach,\emph{
Doctoral Dissertation}, \emph{Hong Kong Polytechnic University.} \end{thebibliography}
\end{document}